\documentclass[11pt, letterpaper]{amsart}
\usepackage{ amssymb}
\usepackage{amsmath}

\def\myQED{\mbox{\rule[0pt]{1.5ex}{1.5ex}}}

\DeclareMathOperator{\esssup}{esssup}

\newtheorem{thm}{Theorem}[section]

\newtheorem{prop}[thm]{Proposition}

\newtheorem{rmk}[thm]{Remark}

\newcommand{\no}{\nonumber}
\numberwithin{equation}{section}

\usepackage{ifpdf}
\ifpdf
\usepackage[pdftex]{graphicx}
 \else
\usepackage[dvips]{graphicx}
 \fi

\usepackage{ifpdf}
\usepackage{color}
\definecolor{webgreen}{rgb}{0,.5,0}
\definecolor{webbrown}{rgb}{.8,0,0}
\definecolor{emphcolor}{rgb}{0.95,0.95,0.95}

\usepackage{hyperref}
\hypersetup{%
          colorlinks=true,
          linkcolor=webbrown,
          filecolor=webbrown,
          citecolor=webgreen,
          breaklinks=true}
\ifpdf \hypersetup{pdftex,
            pdfstartview=FitH, 
            bookmarksopen=true,
            bookmarksnumbered=true
} \else \hypersetup{dvips} \fi

\title{Byzantine Fault Tolerant Distributed Quickest Change Detection }\thanks{ Erhan Bayraktar is supported by the National Science Foundation under grant DMS-11-18673. Lifeng Lai is supported by the National Science Foundation under grant DMS-12-65663.}
\author{Erhan Bayraktar}
\address[Erhan Bayraktar]{Department of Mathematics, University of Michigan, 530 Church Street, Ann Arbor, MI 48104, USA}
\email{erhan@umich.edu}
\author{Lifeng Lai }
\address[Lifeng Lai]{Department of Electrical and Computer Engineering,  Worcester Polytechnic Institute, Worcester, MA 01609, USA}
\email{llai@wpi.edu}

\date{\today}

\begin{document}


\begin{abstract}
We introduce and solve the problem of Byzantine fault tolerant distributed quickest change detection in both continuous and discrete time setups. In this problem, multiple sensors sequentially observe random signals from the environment and send their observations to a control center that will determine whether there is a change in the statistical behavior of the observations. We assume that the signals are independent and identically distributed across sensors. An unknown subset of sensors are compromised and will send arbitrarily modified and even artificially generated signals to the control center. It is shown that the performance of the the so-called CUSUM statistic, which is optimal when all sensors are honest, will be significantly degraded in the presence of even a single dishonest sensor. In particular, instead of in a logarithmically the detection delay grows linearly with the average run length (ARL) to false alarm. To mitigate such a performance degradation, we propose a fully distributed low complexity detection scheme. We show that the proposed scheme can recover the log scaling. We also propose a centralized group-wise scheme that can further reduce the detection delay.
\end{abstract}

\maketitle


\section{Introduction} \label{sec:intro}
In the quickest change detection problem, one observes a sequence of observations, whose probability density function (pdf) might change at an unknown time~\cite{Poor:Book:08}. The goal in this problem is to detect the presence of such a change with a minimum delay under certain false alarm constraints. This type of problem has a broad range of potential applications, such as quality control \cite{Roberts:Tech:66}, security \cite{Tartakovsky:TSP:06,Tartakovsky:TSPS:13}, wireless communications \cite{Lai:GLOBE:08}, and chemical or biological attack detection, etc.


In recent years, motivated by the rapid development of wireless sensor networks, distributed quickest detection has attracted much interest~\cite{Teneketzis:TAC:84,Veeravalli:TIT:01,Mei:TIT:05,Moustakides:ICIF:06,Tartakovsky:SA:08,Tartakovsky:ICIF:03,Tartakovsky:ICIF:06,Tartakovsky:Book:04,Crow:TAES:96}. In such networks, multiple sensors are deployed to facilitate the detection of abnormal activities. These sensors will send observations to a fusion center which will then make final detect decisions~\cite{Willett:TSP:00, Chamberland:JSAC:04}. In the existing distributed quickest change detection setup, the sensors are assumed to always report true observations. In certain applications, especially in the attack detection applications, sensors might be compromised. As a result, the compromised sensors might send modified or even artificially generated observations. In this paper, we consider the quickest detection problem under such \emph{Byzantine attacks}. The problems studied in this paper are related to the Byzantine fault tolerance problems~\cite{Lamport:TPLS:82} in which a component will fail in an arbitrary manner (instead of just stopping functioning). Our goal is to design Byzantine fault tolerant quickest detection schemes.

More specifically, we consider a setup with $N$ sensors, an unknown subset of which might be compromised by adversaries. These sensors observe signals sequentially from the environment. We assume that the signals are independent and identically distributed across sensors. If a sensor is honest, it will send its observations to a control center. If a sensor is compromised, it will send modified observations to the control center. We do not make any assumption on how the attacker will modify its observations. We consider a non-Bayesian setup in which the statistical behavior of the random observations from the environment will change at an unknown but fixed time. The goal of the control center is detect the presence of such a change based on the signals received from all sensors.

If all sensors are honest, it is well known that the cumulative sum (CUSUM) strategy is optimal under Lorden's setup~\cite{Lorden:AMS:71}.
We first show that the performance of CUSUM will be significantly degraded even if only one sensor is compromised. In particular, using a simple attack, the attacker can make the detection delay of CUSUM scale linearly with the average run length (ARL) to false alarm (see Proposition~\ref{prop:degradation}), while the detection delay of CUSUM scales in a logaritmaically with ARL to false alarm when there is no attacker.

To overcome such a performance degradation, we propose a low complexity, low communication overhead detection scheme. In this scheme, we ask each sensor to run a CUSUM locally using its own observed signal, and send one bit of indication to the control center once its CUSUM stops. The control center will raise a final alarm once it receives alarms from at least two sensors. We call this scheme the ``second-alarm'' strategy/scheme. Using tools from order statistics~\cite{David:Book:03}, we provide an upper-bound on the detection delay and an lower-bound on the ARL to false alarm that are valid under arbitrary attack strategies. We show that the derived upper-bound on the detection delay scales logarithmically with that of the lower-bound on the ARL to false alarm; see Theorems~\ref{thm:con} and \ref{thm:dis}. This implies that the proposed second-alarm strategy successfully recovers the logarithmic scaling. Furthermore, this scheme can be implemented in a fully distributed manner.

To further reduce the detection delay, we propose a group-wise strategy. In this group-wise strategy, we divide sensors into three groups. Each group will run a CUSUM using signals received in its group and raise an alarm once its CUSUM stops. The control center then raises the final alarm once at least two groups raise alarms. We show that the detection delay of the group-wise strategy scales in a logarithmically with the ARL to false alarm. Furthermore, the pre-log factor of the group-wise strategy is smaller; see Theorems~\ref{thm:dist-cont} and \ref{thm:dist-disc}. The main intuition for the deduction of the detection delay in the group-wise strategy is that CUSUM in each group can raise an alarm faster than CUSUM in each individual sensor. The disadvantage is that the group-wise strategy is not amenable to distributed implementation.

The remainder of the paper is organized as follows. In Section~\ref{sec:model}, we introduce the models discussed in the paper. In Section~\ref{sec:Brownian}, we study the continuous-time Brownian motion model. In Section~\ref{sec:Discrete}, we discuss the discrete-time model. 

\section{Problem Formulation} \label{sec:model}
Figure~\ref{fig:model} illustrates the model under consideration. Consider a system with $N\geq 3$ sensors that are deployed to monitor the environment. We assume that one sensor might be compromised. The schemes discussed in the paper can be properly modified to handle multiple compromised sensors. We use $n^c$ to denote the index of the compromised sensor. The value of $n^c$ is unknown. 
We assume that the sensors are connected with a fusion center, which will make detection decisions based on the observations from the sensors.
Each sensor $n\in \{1,\cdots, N\}$ observes $\{\xi_{t}^{(n)};t\geq 0\}$. Let $\bar{\xi}_{t}^{(n)}$ be the signal sensor $n$ reports to the fusion center. If sensor $n$ is honest, $\bar{\xi}_{t}^{(n)}=\xi_{t}^{(n)}$. For the compromised sensor $n^c$, $\bar{\xi}_{t}^{(n^c)}$ depends on the attacker's attack strategy $\mathfrak{s}$, which can use the information $\xi_{t}^{(n^c)}$ observed by the attacker. We do not make any assumption on the attack strategy. 
We take $\mathcal{F}=\bigcup_{t \geq 0}\mathcal{F}_{t}$, where $\mathcal{F}_t=\sigma\left\{ \left(\bar{\xi}_{s}^{(1)},\cdots,\bar{\xi}_{s}^{(N)}\right); s\leq t\right\}$. The goal of the fusion center is to detect the presence of a change in the environment based on the processes $(\bar{\xi}_t^{(1)},\cdots,\bar{\xi}_t^{(N)})$.

With a little of abuse of notation, we use the following modified Lorden's performance index as the delay measure
\begin{equation*}
d(T)=\sup\limits_{\tau,n^c, \mathfrak{s}}\esssup E_{\tau,n^c, \mathfrak{s}}[(T-\tau)^+|\mathcal{F}_{\tau}].
\end{equation*}
Here $\esssup$ denotes the essential supremum, the $L^{\infty}$ norm, and $T$ is the stopping rule that the fusion center uses to raise an alarm.
$E_{\tau,n^c,\mathfrak{s}}\{(T-\tau)^+|\mathcal{F}_{\tau}\}$ is the expectation under the measure when the change occurs at time $\tau$, the compromised sensor is $n^c$ and the compromised sensor uses the attack strategy $\mathfrak{s}$. We will use $E_{\tau,0,\phi}[\cdot]$ to denote the expectation under the measure when the change occurs at time $\tau$ and all sensors are honest. We will also use $E_{\infty}$ and $E_{0}$ to denote the expectations with regards to one honest sensor when the change points occurs at $\infty$ and $0$ respectively.

We also have the following average run length (ARL) to false alarm measure
\begin{equation*}
\text{ARL}(T)=\inf\limits_{n^c, \mathfrak{s}} E_{\infty,n^c,\mathfrak{s}}[T].
\end{equation*}
Here $E_{\infty,n^c,\mathfrak{s}}[T]$ denotes the expectation under the measure when the change does not occur, the compromised sensor is $n^c$ and the compromised sensor uses the attack strategy $\mathfrak{s}$.

With the above defined metrics, we aim to solve the following optimization problem
\begin{equation}\label{eq:criterion}
\begin{split}
&\inf\limits_{T}\quad d(T)\\
\text{such that} &\; \text{ARL}(T)\geq \gamma.
\end{split}
\end{equation}
In other words, we aim to design a stopping rule that minimizes the detection delay while making sure that ARL to false alarm is larger than a given threshold $\gamma$.



\begin{figure}[thb]
\centering
\includegraphics[width=0.6 \textwidth]{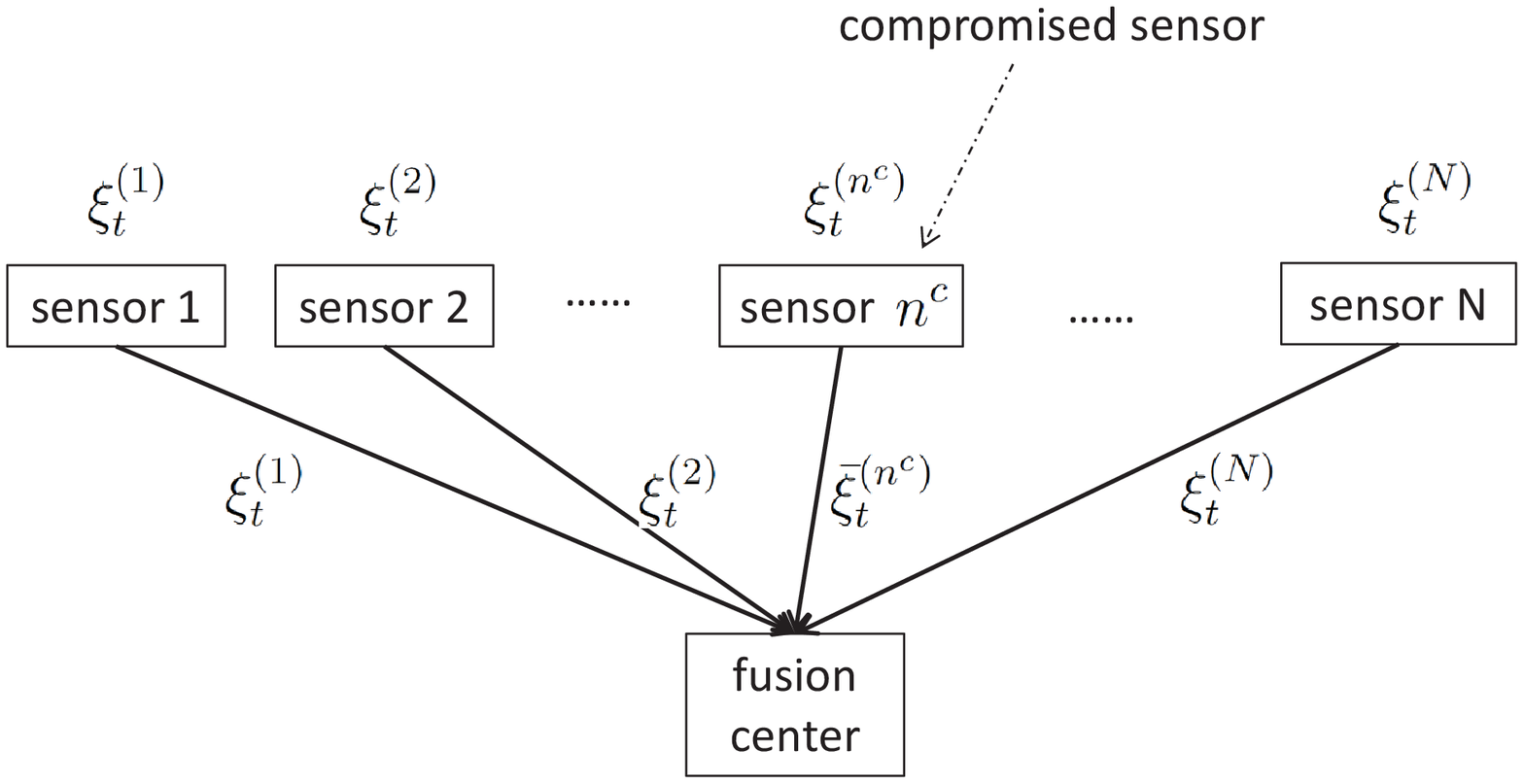}
\caption{Model}
\label{fig:model}
\end{figure}

In this paper, we will focus on two different signal models: a continuous-time Brownian motion model and a discrete-time model.

In the Brownian motion model, we have
\begin{equation*}
\xi_{t}^{(n)}= \mu (t-\tau)^+ + W_t^{(n)},
\end{equation*}
where $\tau$ is an unknown constant at which signals observed at all sensors change, $\mu$ is the signal strength at each node after the change, $\left\{W_t^{(n)}\right\}_{n \in \{1,\cdots N\}}$ are independent standard Brownian motions. 

In the discrete-time model, the time $t$ takes integer values, and we have the following probability density functions (pdf)
\begin{equation*}
\begin{split}
\xi_{t}^{(n)}\sim\left\{\begin{array}{ll}f_{0},& t\leq \tau,\\
                                   f_1,&t>\tau.
                                   \end{array}\right.
\end{split}
\end{equation*}
Here $f_0$ is the pdf of each observation before the change and $f_1$ is the pdf of each observation after the change.




\section{Continuous-time Brownian Motion Model}\label{sec:Brownian}
In this section, we will first discuss the impact on the CUSUM test of the presence of a single attacker. We show that, instead of scaling logarithmically, \emph{the detection delay of CUSUM scales linearly with ARL to false alarm under a simple attack}. We then propose a fully distributed scheme that allows us to recover the logarithmic scaling. We further design a centralized scheme that allows us to further reduce the detection delay.

\subsection{Impact on the CUSUM test}\label{sec:CUSUM}
We first show that the presence of even a single attacker has significant impact on the performance of CUSUM.

When all nodes are honest, i.e., $\bar{\xi}_t^{(n)}=\xi_t^{(n)}$, $\forall n$, the non-Bayesian quickest change detection problem under Lorden's performance index has been well understood. The optimal solution is the continuous-time version of Page's CUSUM stopping rule; see e.g.~\cite{Beibel:AnS:96,Shiryaev:RMS:96}. In particular, if we define
\begin{equation}\label{eq:cusumstat}
y_t=u_t-m_t,
\end{equation}
with
\begin{equation*}
\begin{split}
u_t=&\sum_{n=1}^N\mu\bar{\xi}_t^{(n)}-\frac{1}{2}N\mu^2t,\\
m_t=&\inf\limits_{0\leq s\leq t}u_s,
\end{split}
\end{equation*}
then, the CUSUM stopping rule is
\begin{eqnarray}\label{eq:cusum}
T_{\nu}=\inf\{t\geq 0;y_t\geq \nu\},
\end{eqnarray}
where $\nu$ is a threshold such that
\begin{equation}
\text{ARL}(T_{\nu})=E_{\infty,0,\phi}[T_{\nu}]=\frac{2(e^{\nu}-\nu-1)}{N\mu^2}=\gamma.
\end{equation}
is optimal for \eqref{eq:criterion} for the case when there is no adversary.
For this scheme, the detection delay is
\begin{equation}
d(T_{\nu})=\frac{2(e^{-\nu}+\nu-1)}{N\mu^2}.
\end{equation}

From here, we know that when the ARL constraint $\gamma\uparrow\infty$, we have
\begin{eqnarray}\label{eq:honest}
d(T_{\nu})=\frac{2}{N\mu^2}\left[\log(\text{ARL}(T_{\nu}))+\log\frac{N\mu^2}{2}-1+o(1)\right].
\end{eqnarray}
This means that \emph{the detection delay scales logarithmically with ARL}.

Now suppose one of the sensors is compromised, but the fusion center still employs the CUSUM test in~\eqref{eq:cusum} to detect the change, then a simple attack will significantly degrade the performance of CUSUM. In this attack, the attacker chooses to set $n^c=1$, and generates $\bar{\xi}_t^{(1)}$ according to the following dynamics irrespective of $\xi_t^{(1)}$
\begin{eqnarray}\label{eq:attack}
\bar{\xi}_{t}^{(1)}=N t+ W_t^{(1)},\quad t\geq 0.
\end{eqnarray}
We use $\mathfrak{s}^*$ to denote this particular attack strategy, and use $P_{0,1,\mathfrak{s}^*}$ to denote the probability measure when the change point occurs at time 0, and $P_{\infty,1,\mathfrak{s}^*}$ to denote the probability measure when the change point never occurs.

The impact of this simple attack on the performance of the CUSUM test~\eqref{eq:cusum} is computed in the following proposition.
\begin{prop}\label{prop:degradation}
Under the attack strategy specified in~\eqref{eq:attack}, the ARL to false alarm and the detection delay of CUSUM~\eqref{eq:cusum} are:
\begin{equation*}
\text{ARL}(T_{\nu})\leq E_{\infty,1,\mathfrak{s}^*}[T_{\nu}]=\frac{(e^{-\nu}+\nu-1)}{N\mu^2/2},
\end{equation*}
and
\begin{equation}\label{eq:delay}
d(T_{\nu})\geq E_{0,1,\mathfrak{s}^*}[T_{\nu}]\geq\frac{(e^{-\nu}+\nu-1)}{(3N/2-1)\mu^2},
\end{equation}
respectively.
As a result, instead of logarithmically, \emph{the detection delay of CUSUM scales linearly with ARL to false alarm under a simple attack}.
\end{prop}
\begin{proof}
For $u_t$, we have
\begin{equation*}
u_t=\sum_{n=1}^N\mu\bar{\xi}_t^{(n)}-\frac{1}{2}N\mu^2t=\mu\sum_{n=1}^N\bar{\xi}_t^{(n)}-\frac{1}{2}N\mu^2t.
\end{equation*}
Under probability measure $P_{0,1,\mathfrak{s}^*}$, the process $\xi_t=\sum_{n=1}^N\bar{\xi}_t^{(n)}$ can be written as\begin{equation*}
\xi_t=(2N-1)\mu t+\sqrt{N} W_t,
\end{equation*}
for some  standard Brownian motion $W_t$. As a result, under $P_{0,1,\mathfrak{s}^*}$
\begin{equation*}
u_t=(3N/2-1)\mu^2 t+\sqrt{N}\mu W_t.
\end{equation*}

Now following the techniques introduced in~\cite{Moustakides:AnS:04}, we will compute $E_{0,1,\mathfrak{s}^*}[T_{\nu}]$.
Let us denote $$f(y)=y+e^{-y}-1.$$ It is easy to check that
\begin{equation*}
f^{'}(y)=1-e^{-y}, \; f^{''}(y)=e^{-y}, \; f^{'}(y)+f^{''}(y)=1, \; f^{'}(0)=f(0)=0,
\end{equation*}
and $f^{''}(y)>0$, $\forall y.$

Applying It\^{o}'s rule, we obtain
\begin{equation*}
\begin{split}
f(y_t)-f(0)&=\int_0^tf^{'}(y_t)(du_t-dm_t)+\int_0^{t}\frac{N}{2}\mu^2f^{''}(y_t)dt\no\\
=&\int_0^t(3N/2-1)\mu^2f^{'}(y_t)dt+\int_0^t\sqrt{N}\mu f^{'}(y_t) dW_t-\int_0^tf^{'}(y_t)dm_t+\int_0^t\frac{N}{2}\mu^2f^{''}(y_t)dt\no.
\end{split}
\end{equation*}
Using $f'(0)=0$ we obtain that $-\int_0^tf^{'}(y_t)dm_t=0$. Hence, we are left with
\begin{eqnarray}\label{eq:cuat}
f(y_t)-f(0)&=&\int_0^t(3N/2-1)\mu^2f^{'}(y_t)dt+\int_0^t\sqrt{N}\mu f^{'}(y_t) dW_t+\int_0^t\frac{N}{2}\mu^2f^{''}(y_t)dt\no\\
&\overset{(a)}\leq&\int_0^t(3N/2-1)\mu^2(f^{'}(y_t)+f^{''}(y_t))dt+\int_0^t\sqrt{N}\mu f^{'}(y_t) dW_t\no\\
&\overset{(b)}=&(3N/2-1)\mu^2t+\int_0^t\sqrt{N}\mu f^{'}(y_t) dW_t,
\end{eqnarray}
in which $(a)$ in the above equation is due to the fact that $f^{''}(y)>0$, and $(b)$ is due to the fact that $f^{'}(y)+f^{''}(y)=1$. Now, \eqref{eq:delay} easily follows by evaluating~\eqref{eq:cuat} at $T_{\nu}$ and noticing the second term in~\eqref{eq:cuat} is 0. The computation of $E_{\infty,1,\mathfrak{s}^*}[T_{\nu}]$ is similar to that of~\cite{Moustakides:AnS:04}. In particular, one can define $$g(y)=-y+e^{y}-1,$$ and then follow the similar steps as above.
\end{proof}

\subsection{A Fully Distributed Byzantine Attack Resistant Change Detection Scheme}\label{sec:dis}
In this section, we propose an intuitive scheme that enables us to recover the logarithmic scaling. In this scheme, the fusion center first runs $N$ independent CUSUMs, $T_h^{(n)}$, $n=1,\cdots, N$, one for each signal $\bar{\xi}_t^{(n)}$ received from sensor $n$:
\begin{equation*}
T_h^{(n)}=\inf\{t\geq 0;y_t^{(n)}\geq h\},
\end{equation*}
in which
\begin{equation*}
y_t^{(n)}=\mu\bar{\xi}_t^{(n)}-\frac{1}{2}\mu^2t-\inf\limits_{s\leq t} \left(\mu\bar{\xi}_s^{(n)}-\frac{1}{2}\mu^2s\right).
\end{equation*}
We call $T_h^{(n)}$ the time at which sensor $n$ raises an alarm. In our scheme, the fusion center will ignore the first alarm from sensors, and will raise the final alarm once a second sensor raises an alarm. Let $T_h$ be the time when the fusion center raises the final alarm. From the description above, we know that $T_h$ is the same as the second order statistics of $N$ i.i.d. random variables $T_h^{(1)},\cdots,T_h^{(N)}$. Following the notation in order statistics~\cite{David:Book:03}, we use $T_{(2), N}$ to denote the second order statistics of $N$ random variables $T_h^{(1)},\cdots,T_h^{(N)}$. As a result, the time at which the fusion center raises an alarm is $T_h \triangleq T_{(2), N}$.

Now, we analyze the detection delay $d(T_h)$ and ARL to false alarm $\text{ARL}(T_h)$. We have the following result:
\begin{thm}\label{thm:con}
The detection delay of $T_h$ scales with ARL at most logarithmically:
\begin{equation*}
d(T_h)\leq \frac{4}{\mu^2}\left(\log (\text{ARL}(T_h))+\log\frac{(N-1)\mu^2}{2}-1+o(1)\right).
\end{equation*}
\end{thm}
\begin{proof}
We first provide a lower bound on ARL to false alarm that is valid for any attacker's strategy. From the attacker's perspective, to reduce ARL to false alarm of the proposed scheme, it should make the compromised sensor to raise an alarm as soon as possible. The proposed scheme will then raise a false alarm when one of the remaining $N-1$ honest sensors raises an alarm under the probability measure $P_{\infty, N, \mathfrak{s} }$. Since all sensors are identical, for the purpose of performance evaluation, we can simply assume that sensor $N$ is compromised. Let $T_{(1),N-1}$ be the first order statistics of $T_h^{(1)},\cdots,T_h^{(N-1)}$, namely
$T_{(1),N-1}\triangleq \min\left\{T_h^{(1)},\cdots,T_h^{(N-1)}\right\}$.
From the discussion above, we have $\text{ARL}(T_h)\geq E_{\infty,N,\mathfrak{s}}\{T_{(1),N-1}\}$.

In the following, we use an estimate of $E_{\infty;N;\mathfrak{s}}\{T_{(1),N-1}\}$ due to~\cite{Hadjiliadis:TIT:09}. For completeness and in preparation for other estimates, we present the main steps of proof from~\cite{Hadjiliadis:TIT:09} as well.
Since $T_{(1),N-1}$ is a positive random variable, we have
\begin{equation*}
\begin{split}
\text{ARL}(T_h)\geq& E_{\infty,N,\mathfrak{s}}\{T_{(1),N-1}\}\\
                       =&\int_0^{\infty}P_{\infty,N,\mathfrak{s}}(T_{(1),N-1}\geq t)dt\\
                       =&\int_0^{\infty}P_{\infty,N,\mathfrak{s}}\left(\min\left\{T_h^{(1)},\cdots,T_h^{(N-1)}\right\}\geq t\right)dt\\
                       =&\int_0^{\infty}\left[P_{\infty}(T_h^{(1)}\geq t)\right]^{N-1}dt.
\end{split}
\end{equation*}
The last equality is due to the fact that for those sensors not affected by the attacker, $T_h^{(n)}$ are i.i.d. random variables.

Next, we provide an upper-bound on the detection delay that holds for any strategy of the attacker. Clearly, to increase the detection delay of the proposed scheme, the attacker should use a strategy that does not raise an alarm. In this case, the fusion center will raise an alarm only when at least two honest sensors raise alarms under the probability measure $P_{0}$. Again, for the performance evaluation purpose, we can simply assume that sensor $N$ is compromised.


Let $T_{(2),N-1}$ be the second order statistics of $T_h^{(1)},\cdots,T_h^{(N-1)}$. Based on the discussion above, we have
\begin{equation*}
\begin{split}
& d(T_h)\leq E_{0,N,\mathfrak{s}}\{T_{(2),N-1}\}\\
                       =&\int_0^{\infty}P_{0,N,\mathfrak{s}}(T_{(2),N-1}\geq t)dt\\
                       =&\int_0^{\infty}P_{0,N,\mathfrak{s}}\left\{\text{at most 1 of $T_h^{(1)},\cdots,T_h^{(N-1)}$ is less than $t$}\right \}dt\\
                       =&\int_0^{\infty}P_{0,N,\mathfrak{s}}\left\{\text{ none of $T_h^{(1)},\cdots,T_h^{(N-1)}$ is less than $t$}\right \}dt\no\\
                       &+\int_0^{\infty}P_{0,N,\mathfrak{s}}\left\{\text{ 1 of $T_h^{(1)},\cdots,T_h^{(N-1)}$ is less than $t$}\right \}dt\\
                       =&\int_0^{\infty}P_{0,N,\mathfrak{s}}\left(\min\left\{T_h^{(1)},\cdots,T_h^{(N-1)}\right\}\geq t\right)dt
                   +\int_0^{\infty}P_{0,N,\mathfrak{s}}\left\{\text{ 1 of $T_h^{(1)},\cdots,T_h^{(N-1)}$ is less than $t$}\right \}dt\\
                       =&\int_0^{\infty}\left[P_{0}(T_h^{(1)}\geq t)\right]^{N-1}dt
                       +\int_0^{\infty}(N-1)(1-P_{0}(T_h^{(1)}\geq t))\left[P_{0}(T_h^{(1)}\geq t)\right]^{N-2}dt\label{eq:N-1},
                   \end{split}
\end{equation*}
in which the first inequality is due to the non-negativity of $y_t^{(n)}$ which implies the worst case delay will occur when $y_t^{(n)}, n=1,\cdots, N-1$ are 0 at the time of the change ($\tau=0$ satisfies the condition).

To proceed, let us recall the following result from~\cite{Ismail:JAP:04}:
\begin{eqnarray}
P_{0}(T_{h}^{(n)}\geq t)&=&2e^{\frac{h}{2}}\sum\limits_{k\geq 1} u(\phi_{k})e^{-\frac{\mu^2t}{8\cos^2(\phi_k)}},\label{eq:p0}\\
P_{\infty}(T_{h}^{(n)}\geq t)&=&2e^{-\frac{h}{2}}\sum\limits_{k\geq 1} u(\theta_{k})e^{-\frac{\mu^2t}{8\cos^2(\theta_k)}}+2e^{-\frac{h}{2}}v(\eta)e^{-\frac{\mu^2t}{8\cosh^2(\eta)}},\label{eq:p1}
\end{eqnarray}
where
\begin{equation*}
u(x)=\frac{\sin^3x}{x-\sin x\cos x}, \quad v(x)=\frac{\sinh^3x}{\sinh x\cosh x-x},
\end{equation*}
and the constants $\phi_k$, $ \theta_k$, and $\eta$ are defined as the solutions of
\begin{equation*}
\tan \phi_k=-\frac{2\phi_k}{h}<0,\quad \tan \theta_k=\frac{2\theta_k}{h}>0, \quad \text{and}\; \tanh \eta=\frac{2\eta}{h}>0,
\end{equation*}
respectively. With these results, we continue our estimations.
\begin{equation}\label{eq:RLS-est}
\begin{split}
\text{ARL}(T_h)\geq&\int_0^{\infty}\left[P_{\infty}(T_h^{(1)}\geq t)\right]^{N-1}dt\\
                       =&\int_0^{\infty}\left[2e^{-\frac{h}{2}}\sum\limits_{k\geq 1} u(\theta_{k})e^{-\frac{\mu^2t}{8\cos^2(\theta_k)}}+2e^{-\frac{h}{2}}v(\eta)e^{-\frac{\mu^2t}{8\cosh^2(\eta)}}\right]^{N-1}dt\\
                       \overset{(a)}=&\frac{2}{(N-1)\mu^2}[e^{h}+\text{``lower exponents''}],
 \end{split}
\end{equation}
in which (a) in \eqref{eq:RLS-est}  is true due to (80) of~\cite{Hadjiliadis:TIT:09}.

As for $E_{0,N,\mathfrak{s}}\{T_{(2),N-1}\}$, it is obvious that $P_{0,N,\mathfrak{s}}(T_{(2),N-1}\geq t)$ is a non-increasing function of $N$. Hence from~\eqref{eq:N-1}, we have
\begin{equation}\label{eq:delay-est}
\begin{split}
d(T_h)\leq& E_{0,N,\mathfrak{s}}\{T_{(2),N-1}\}\\
=&\int_0^{\infty}\left[P_{0}(T_h^{(1)}\geq t)\right]^{N-1}dt\no\\
                       &+\int_0^{\infty}(N-1)(1-P_{0}(T_h^{(1)}\geq t))\left[P_{0}(T_h^{(1)}\geq t)\right]^{N-2}dt\\
                    &\leq  \int_0^{\infty}[P_{0}(T_h^{(1)}\geq t)]^{2}dt
                    + \int_0^{\infty}2(1-P_{0}(T_h^{(1)}\geq t))P_{0}(T_h^{(1)}\geq t)dt\\
                    =& 2\int_0^{\infty}P_{0}(T_h^{(1)}\geq t) dt-\int_0^{\infty} [P_{0}(T_h^{(1)}\geq t)]^2 dt\\
                    =& \frac{4(e^{-h}+h-1)}{\mu^2}-\int_0^{\infty} [P_{0}(T_h^{(1)}\geq t)]^2 dt\\
                    \leq& \frac{4(e^{-h}+h-1)}{\mu^2}.
\end{split}
\end{equation}
It now follows that $d(T_h)$ grows logarithmically with $\text{ARL}(T_h)$.
\end{proof}

\begin{rmk}
The bounds derived in the proof are valid for any strategy of the attacker, who can use an arbitrarily complicated attack strategy and can arbitrarily change its behavior over time.
\end{rmk}
\begin{rmk}
This scheme is amenable to \emph{a fully distributed implementation with very little communication overhead}. In particular, each sensor $n$ can run $T_h^{(n)}$ locally, and send an alarm (one bit) to the fusion center once $y_t^{(n)}$ is larger than the threshold $h$. The fusion center will raise a final alarm after receiving the second alarm from the sensors.
\end{rmk}
\begin{rmk}
The proposed scheme also works for the scenario in which the Brownian motion observed by each sensor has different drift and volatility coefficients. The analysis follows the same steps as above but with more complicated computations.
\end{rmk}
\begin{rmk}
The proposed second-alarm strategy can be modified for the case with more than 1 sensor or the case with an unknown number of sensors might be compromised, as long as we know an upper-bound $N_{max}$ on the maximum number of sensors that might be compromised and $N_{max}<N/2$. In particular, the fusion center can employ $N_{max}+1$-alarm strategy, in which the fusion center will raise the final alarm once it receives $N_{max}+1$ alarms from the sensors. The condition $N_{max}<N/2$ is necessary for the revised strategy, as the revised strategy will never raise an alarm if $N_{max}>N/2$ and the attacker simply make the affected sensors all silent.
\end{rmk}
\subsection{A Group-Wise Scheme}\label{sec:group}
The scheme presented in Section~\ref{sec:dis} can be implemented in a fully distributed manner and the detection delay scales with ARL in a $\log$ order. However, compared with the case in which there is no dishonest sensor or the case in which the identity of the dishonest sensor is known, the pre-log factor of the scheme in Section~\ref{sec:dis} is larger, which implies that the delay increases faster with the ARL constraint. In particular, if there is no dishonest sensor, from~\eqref{eq:honest}, we know the pre-log factor is $2/(N\mu^2)$. If we know the identity of the dishonest sensor, then the problem is reduced to a case with $N-1$ honest sensors. Again from~\eqref{eq:honest}, we know that the pre-log factor is $2/((N-1)\mu^2)$. On the other hand, the pre-log factor of the fully distributed scheme discussed in Section~\ref{sec:dis} is $4/\mu^2$. In this section, we propose a modified scheme that achieves a better slope that decreases with increasing $N$.

In the modified scheme, we evenly divide these $N$ sensors into 3 groups, each with $N/3$ sensors. Here we assume that $N$ can be divided by 3. The scheme can be easily modified if $N$ cannot be divided by 3. We use $\mathcal{G}_i, i=1,2,3$ to denote each group of sensors. For each group $\mathcal{G}_i$, we run the following scheme:
\begin{equation*}
T_h^{\mathcal{G}_i}=\inf\{t\geq 0;y_t^{\mathcal{G}_i}\geq h\},
\end{equation*}
in which
\begin{equation*}
y_t^{\mathcal{G}_i}=u_t^{\mathcal{G}_i}-m_t^{\mathcal{G}_i}
\end{equation*}
with
\begin{equation*}
\begin{split}
u_t^{\mathcal{G}_i}=&\sum_{n\in \mathcal{G}_i}\mu\bar{\xi}_t^{(n)}-\frac{1}{2}\frac{N}{3}\mu^2t, \quad \text{and} \quad m_t^{\mathcal{G}_i}=\inf\limits_{0\leq s\leq t}u_s^{\mathcal{G}_i}.
\end{split}
\end{equation*}

In our scheme, the fusion center will raise an alarm when at least two out of these three groups raise alarms. In other words, $T^G_h$, the time when the fusion center raises an alarm, is $T^{G}_{(2),3}$, the second order statistic of three random variables $T_h^{\mathcal{G}_1},T_h^{\mathcal{G}_2},T_h^{\mathcal{G}_3}$.

Now, we analyze the detection delay $d(T^{G}_h)$ and ARL to false alarm $\text{ARL}(T^{G}_h)$. We have the following result:
\begin{thm}\label{thm:dist-cont}
The detection delay $T^G_h$ scales logarithmically with ARL in the following manner:
\begin{equation*}
d(T^{G}_h)\leq \frac{12}{\mu^2 N}\left(\log (\text{ARL}(T^{G}_h))+\log\frac{N\mu^2}{3}-1+o(1)\right).
\end{equation*}
\end{thm}
\begin{proof}
The proof follows that of Theorem~\ref{thm:con} with only minor modification. In particular, in the group-wise scheme,
$u_t^{\mathcal{G}_i}$ has a drift of $-\mu^2 N/6$ before the change, $\mu^2 N/6$ after the change, and its standard deviation is $\mu\sqrt{N/3}$.
Following the steps in the proof of Theorem~\ref{thm:con}, we obtain the estimate
\begin{equation*}
\text{ARL}(T^{G}_h)\geq E_{\infty,N,\mathfrak{s}}\{T^{G}_{(1),2}\}=\frac{3}{\mu^2 N}\left[e^h+\text{``lower exponents''}\right].
\end{equation*}
Similarly,
\begin{equation*}
\begin{split}
d(T^{G}_h)\leq E_{0,N,\mathfrak{s}}\{T^{G}_{(2),2}\} \leq  2E_{0,N,\mathfrak{s}}\{T_{h}^{\mathcal{G}_1}\}
=\frac{2}{\mu^2N/6}(e^{-h}+h-1).
\end{split}
\end{equation*}
The statement of the theorem easily follows from the above two estimates.
\end{proof}
\begin{rmk}
This modified scheme achieves a smaller detection delay than the one discussed in Section~\ref{sec:dis}. However, this modified scheme is a centralized scheme as the fusion center needs signals from all sensor in each group.
\end{rmk}
\begin{rmk}
One may wonder whether we can further divide the sensors into two groups instead of three. There are some potential issues with dividing the sensors into two groups. If one employs the the second-alarm scheme, the attacker can make the detection delay arbitrarily long by generating a Brownian motion with a sufficiently small negative drift. In this way, the combined signal in the group containing the attacker will not raise an alarm (or will raise an alarm very late) even if there is a true change in the environment. If, instead of using the second-alarm scheme, we opt to use the usual CUSUM scheme of Section~\ref{sec:CUSUM}, then the attacker can make ARL to false alarm very short by generating a Brownian motion with a sufficiently large positive drift.
\end{rmk}

\section{Discrete-time Model}\label{sec:Discrete}
In this section, we discuss the discrete-time model for quickest change detection with compromised sensors. We follow the same structure as in the Brownian motion model. In this section, to follow conventional notation, we will use $k$ to denote time $t$.

\subsection{The CUSUM Test}
Similar to the Brownian motion case, when all nodes are honest, the non-Bayesian quickest change detection problem under Lorden's performance index has been well understood. The optimal solution is Page's CUSUM stopping rule~\cite{Moustakides:AoS:86}. In particular, define
\begin{eqnarray}\label{eq:cusumstat}
y_{k}=u_{k}-m_{k},
\end{eqnarray}
with
\begin{equation*}
\begin{split}
u_{k}=&\sum_{l=1}^k\sum_{n=1}^N\log\left\{\frac{ f_1(\xi_l^{(n)})}{f_0(\xi_l^{(n)})}\right\}\triangleq \sum_{l=1}^k\sum_{n=1}^N Z_l^{(n)},\\
m_{k}=&\min\limits_{1\leq j\leq k}u_{j}.
\end{split}
\end{equation*}
Here, $k$ is the time index. We assume that the second moment of $Z^{(n)}$s for $n=1,\cdots,N$ under $f_1$ are finite. The CUSUM stopping rule is
\begin{eqnarray}\label{eq:cusumd}
T_{\nu}=\inf\{k\geq 1;y_{k}\geq \nu\}
\end{eqnarray}
where $\nu$ is a threshold such that
\begin{equation*}
E_{\infty,0,\phi}[T_{\nu}]=\gamma.
\end{equation*}

The performance evaluation of the discrete-time CUSUM is more involved than that of the Brownian motion model, mainly because of the possibility of overshoot. Define
\begin{eqnarray}
\kappa&=&\lim\limits_{\nu\rightarrow\infty}E_{0,0,\phi}\{y_{T_{\nu}}-\nu\},\label{eq:kappa}\\
\beta&=&E_{0,0,\phi}\{m_{\infty}\},\label{eq:beta}\\
R&=&\lim\limits_{\nu\rightarrow\infty}E_{0,0,\phi}\{\exp[-u_{\eta_{\nu}}-\nu]\},\label{eq:r}
\end{eqnarray}
with
\begin{equation*}
\eta_{\nu}=\inf\{k:u_{k}\geq \nu\}.
\end{equation*}

From~\cite{Tartakovsky:CDC:05} one has
\begin{equation*}
\text{ARL}(T_{\nu})=E_{\infty,0,\phi}[T_{\nu}]=\frac{1}{R^2ND(f_1||f_0)}e^{\nu}[1+o(1)],
\end{equation*}
in which $D(f_1||f_0)$ is the Kullback-Leibler distance between $f_1$ and $f_0$.
In addition, from~\cite{Tartakovsky:CDC:05}, one has
\begin{equation*}
d(T_{\nu})=E_{0,0,\phi}[T_{\nu}]=\frac{1}{ND(f_1||f_0)}(\nu+\beta+\kappa)+o(1).
\end{equation*}




\subsection{A Fully Distributed Byzantine Attack Resistant Change Detection Scheme}

Here, we show that a scheme modified from the one discussed in Section~\ref{sec:dis} enables us to recover the logarithmic scaling between the detection delay and average run length. In this scheme, the fusion center first runs $N$ independent CUSUM $T_h^{(n)}$s, one for each signal $\bar{\xi}_k^{(n)}$ received from sensor $n$:
\begin{equation*}
T_h^{(n)}=\inf\{k\geq 1;y_k^{(n)}\geq h\},
\end{equation*}
in which
\begin{eqnarray}\label{eq:cusumstatd}
y_{k}^{(n)}=u_{k}^{(n)}-m_{k}^{(n)},
\end{eqnarray}
with
\begin{equation*}
\begin{split}
u_{k}^{(n)}=&\sum_{l=1}^k\log\left\{\frac{ f_1(\bar{\xi}_l^{(n)})}{f_0(\bar{\xi}_l^{(n)})}\right\}\triangleq \sum_{l=1}^k Z_l^{(n)},\\
m_{k}^{(n)}=&\min\limits_{1\leq j\leq k}u_{j}^{(n)}.
\end{split}
\end{equation*}


Similar to the Brownian motion model, the fusion center will raise the final alarm when two alarms are raised from these $N$ CUSUMs. That is, $T_h$, the time the fusion center raises the final alarm, is the same as $T_{(2), N}$, the second order statistics of $N$ random variables $T_h^{(1)},\cdots,T_h^{(N)}$.

Now, we analyze the detection delay $d(T_h)$ and ARL to false alarm, which we denote by $\text{ARL}(T_h)$.
\begin{thm}\label{thm:dis}
The detection delay of $T_h$ scales logarithmically with ARL :
\begin{equation*}
d(T_h)\leq \frac{2}{D(f_1||f_0)}\left(\log\text{ARL}(T_h)+\log\frac{(1-\exp(1-N))R_1^2D(f_1||f_0)}{\exp(1-N)}+\beta_1+\kappa_1\right)+o(1).
\end{equation*}
\end{thm}
\begin{proof} The proof follows a structure similar to the Brownian motion case.
We first provide a lower bound for ARL  to false alarm  for any strategy the attacker might use.
From the attacker's perspective, to reduce the ARL to false alarm of the proposed scheme, it should make the compromised sensor to raise an alarm as soon as possible. The proposed scheme will then raise a false alarm when one of the remaining $N-1$ sensors raises an alarm under the probability measure $P_{\infty}$. Since all sensors are identical, for the evaluation of the performance purpose, we can simply assume sensor $N$ is compromised. Again, let $T_{(1),N-1}\triangleq \min\left\{T_h^{(1)},\cdots,T_h^{(N-1)}\right\}$ be the first order statistics of $T_h^{(1)},\cdots,T_h^{(N-1)}$.
From the discussion above, we have
$\text{ARL}(T_h)\geq E_{\infty,N,\mathfrak{s}}\{T_{(1), N-1}\}$.
Following the same steps as those in the Brownian motion case, we have

\begin{equation}\label{eq:barhd}
\begin{split}
\text{ARL}(T_h)\geq E_{\infty,N,\mathfrak{s}}\{T_{(1), N-1}\} =\sum\limits_{k=1}^{\infty}\left[P_{\infty}(T_h^{(1)}\geq k)\right]^{N-1}.
\end{split}
\end{equation}

Now, we provide an upper-bound for the detection delay under any attack. Clearly, to increase the detection delay of the proposed scheme, the attacker should use a strategy that does not raise an alarm. In this case, the fusion center will raise an alarm only when at least two honest sensors raise alarms under the probability measure $P_{0}$. Since all sensors are identical, for the evaluation of the performance purpose, we can simply assume sensor $N$ is compromised.

Let $T_{(2),N-1}$ be the second order statistics of $T_h^{(1)},\cdots,T_h^{(N-1)}$. Based on the discussion above, and follow the same steps as those in the Brownian motion case, we have

\begin{equation*}
\begin{split}
d(T_h)\leq& E_{0,N,\mathfrak{s}}\{T_{(2),N-1}\}\\
                       =&\sum\limits_{k=1}^{\infty}P_{0,N,\mathfrak{s}}(T_{(2),N-1}\geq k)\\
                       =&\sum\limits_{k=1}^{\infty}\left[P_{0}(T_h^{(1)}\geq k)\right]^{N-1}
                       +\sum\limits_{k=1}^{\infty}(N-1)(1-P_{0}(T_h^{(1)}\geq k))\left[P_{0}(T_h^{(1)}\geq k)\right]^{N-2}.\label{eq:N-1d}
                    \end{split}
                       \end{equation*}

Unlike the Brownian motion case, closed form expressions for $P_{0}(T_h^{(1)}\geq k)$ and $P_{\infty}(T_h^{(1)}\geq k)$ are unknown. However, it has been shown in~\cite{Khan:SA:95} that a properly normalized $T_h^{(1)}$ has a geometric distribution under $P_{\infty}$. In particular, letting
\begin{equation*}
C=\frac{1}{R_1^2D(f_1||f_0)}e^{h}[1+o(1)],
\end{equation*}
we have
\begin{equation*}
\lim\limits_{h\rightarrow \infty} P_{\infty}\left\{\frac{T_h^{(1)}}{C}\geq k\right\}=\exp(-k).
\end{equation*}
Here, $R_1$ is the quantity corresponds to~\eqref{eq:r} when there is only one observation sequence.

We continue the computation.
\begin{equation*}
T_{(1), N-1}=\min\left\{T_h^{(1)},\cdots,T_h^{(N-1)}\right\}=C\min\left\{T_h^{(1)}/C,\cdots,T_h^{(N-1)}/C\right\}.
\end{equation*}

Following~\eqref{eq:barhd}, we have
\begin{equation*}
\begin{split}
\text{ARL}(T_h)\geq& E_{\infty,N,\mathfrak{s}}\{T_{(1), N-1}\}\\
=&C\sum\limits_{k=1}^{\infty}\left[P_{\infty}(T_h^{(1)}/C\geq k)\right]^{N-1}\\
                       =&C\sum\limits_{k=1}^{\infty}\exp(-(N-1)k)\\
                       =&\frac{\exp(1-N)}{1-\exp(1-N)}C\\
                       =&\frac{\exp(1-N)}{(1-\exp(1-N))R_1^2D(f_1||f_0)}e^{h}[1+o(1)].
 \end{split}
                       \end{equation*}

Now, for $d(T_h)$, as the Brownian motion case, $P_{0}(T_{(2),N-1}\geq k)$ is a non-increasing function of $N$. Hence from~\eqref{eq:N-1d}, we have
\begin{equation*}
\begin{split}d(T_h)\leq& E_{0,N,\mathfrak{s}}\{T_{(2),N-1}\}\\
      =&\sum\limits_{k=1}^{\infty}\left[P_{0}(T_h^{(1)}\geq k)\right]^{N-1}
      +\sum\limits_{k=1}^{\infty}(N-1)(1-P_{0}(T_h^{(1)}\geq k))\left[P_{0}(T_h^{(1)}\geq k)\right]^{N-2}\\
                    \leq &\sum\limits_{k=1}^{\infty}\left[P_{0}(T_h^{(1)}\geq k)\right]^2+ \sum\limits_{k=1}^{\infty}2(1-P_{0}(T_h^{(1)}\geq k))P_{0}(T_h^{(1)}\geq k)\no\\
                    =& \sum\limits_{k=1}^{\infty}2P_{0}(T_h^{(1)}\geq k) -\sum\limits_{k=1}^{\infty} P_{0}^2(T_h^{(1)}\geq k)\no\\
                    =& \frac{2}{D(f_1||f_0)}(h+\beta_1+\kappa_1)+o(1)-\sum\limits_{k=1}^{\infty} P_{0}^2(T_h^{(1)}\geq k)\no\\
                    \leq& \frac{2}{D(f_1||f_0)}(h+\beta_1+\kappa_1)+o(1),
 \end{split}
                       \end{equation*}
in which $\kappa_1$ and $\beta_1$ are the quantities corresponding to~\eqref{eq:kappa} and~\eqref{eq:beta} when there is only one observation sequence.

As a result, we see that $d\{T_h\}$ grows logarithmically with $\text{ARL}\{T_h\}$.
\end{proof}
\subsection{Group-Wise Scheme}
Similar to the Brownian motion case, we can design a group-wise scheme that has a smaller detection delay. In particular, we evenly divide these $N$ sensors into 3 groups, each with $N/3$ sensors. For each group $\mathcal{G}_i$, we run the following scheme:
\begin{equation*}
T_h^{\mathcal{G}_i}=\inf\{k\geq 1;y_k^{\mathcal{G}_i}\geq h\},
\end{equation*}
in which
\begin{equation*}
y_k^{\mathcal{G}_i}=u_k^{\mathcal{G}_i}-m_k^{\mathcal{G}_i}
\end{equation*}
with
\begin{equation*}
\begin{split}
u_k^{\mathcal{G}_i}=\sum_{l=1}^k \sum_{n\in \mathcal{G}_i}\log\left\{\frac{ f_1(\xi_l^{(n)})}{f_0(\xi_l^{(n)})}\right\}, \quad m_k^{\mathcal{G}_i}=\min\limits_{1\leq j\leq k}u_j^{\mathcal{G}_i}.
\end{split}
\end{equation*}

In our scheme, the fusion center will raise an alarm when at least two out of these three groups raise alarms. In other words, $T^G_h$, the time when the fusion center raises an alarm, is $T^{G}_{(2),3}$, the second order statistic of three random variables $T_h^{\mathcal{G}_1},T_h^{\mathcal{G}_2},T_h^{\mathcal{G}_3}$. Regarding this group-wise scheme, we have the following result.
\begin{thm}\label{thm:dist-disc}
For the group-wise scheme, the detection delay scales with ARL to false alarm in the following manner:
\begin{equation*}
\begin{split}
d(T^G_h)\leq\frac{2}{D(f_1||f_0)N/3}\left(\log\text{ARL}(T^G_h)+\log\frac{(1-\exp(-1))R_G^2D(f_1||f_0)N/3}{\exp(-1)}+\beta_G+\kappa_G\right)+o(1),
\end{split}
\end{equation*}
in which $\beta_G$, $\kappa_G$ and $R_G$ are quantities corresponding to~\eqref{eq:kappa},~\eqref{eq:beta} and~\eqref{eq:r} when there are $N/3$ sequences.
\end{thm}
\begin{proof}
The proof follows closely with that of Theorem~\ref{thm:dis}, with the difference that when all sensors in $G_1$ are honest, for $T_h^{\mathcal{G}_1}$, we have
\begin{eqnarray}
d(T_h^{\mathcal{G}_1})=\frac{1}{D(f_1||f_0)N/3}(h+\beta_G+\kappa_G)+o(1),
\end{eqnarray}
and $T_h^{\mathcal{G}_1}/C_G$ is geometrically distributed under $P_{\infty,0,\phi}$, in which
\begin{eqnarray}
C_G=\frac{1}{R_G^2D(f_1||f_0)N/3}e^{h}[1+o(1)],
\end{eqnarray}
and
\begin{eqnarray}
\lim\limits_{h\rightarrow \infty} P_{\infty,0,\phi}\left\{\frac{T_h^{\mathcal{G}_1}}{C_G}\geq k\right\}=\exp(-k).
\end{eqnarray}

Following the same steps in the proof of Theorem~\ref{thm:dis}, we obtain
\begin{eqnarray}
\text{ARL}(T^G_h)\geq \frac{\exp(-1)}{(1-\exp(-1))R_G^2D(f_1||f_0)N/3}e^{h}[1+o(1)].
\end{eqnarray}
and
\begin{eqnarray}
d(T^G_h)\leq\frac{2}{D(f_1||f_0)N/3}(h+\beta_G+\kappa_G)+o(1).
\end{eqnarray}
\end{proof}

\section{Numerical Examples}\label{sec:num}

In this section, we present several examples to illustrate the analytical results presented in this paper. In generating these figures, we consider the worst case scenario, in which we give the attacker additional information and allow the attacker to change its attack strategy. In particular, when we run the simulation to evaluate ARL, we have the attacher make the affected sensor raise an alarm immediately. When we run the simulation to evaluate delay, we have the attacker keep the affected sensor silent. Hence, the curves we presented here are the worst case curves.

\begin{figure}[thb]
\centering
\includegraphics[width=0.6 \textwidth]{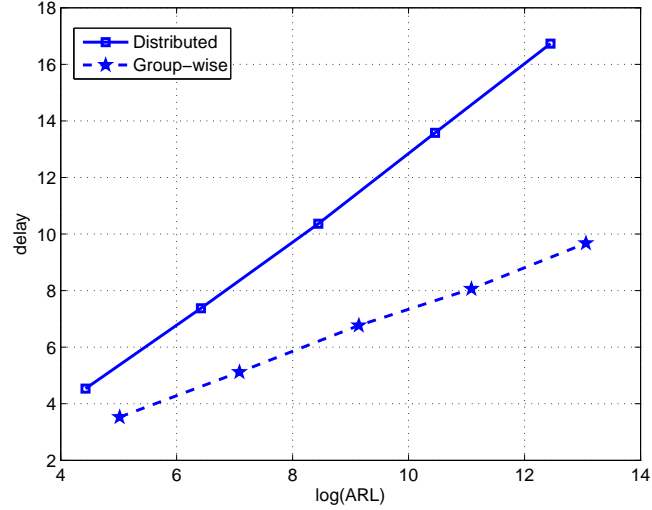}
\caption{Detection delay vs ARL for the continuous-time Brownian motion case}
\label{fig:continuous}
\end{figure}
Figure~\ref{fig:continuous} illustrates the relationship between the detection delay and ARL for the continuous-time Brownian motion case. In generating this figure, we set the number of sensors $N$ to be 9, the drift after the change $\mu$ to be $1$. From the figure, one can see that the detection delay scales logarithmically with ARL for both distributed and group-wise schemes. Furthermore, the delay and its slope of the group-wise scheme are smaller than those of the distributed scheme. This illustrates the benefits of the group-wise scheme. However, as discussed in the paper, the communication overhead of the group-wise scheme is significantly higher than that of the fully distributed scheme.

\begin{figure}[thb]
\centering
\includegraphics[width=0.6 \textwidth]{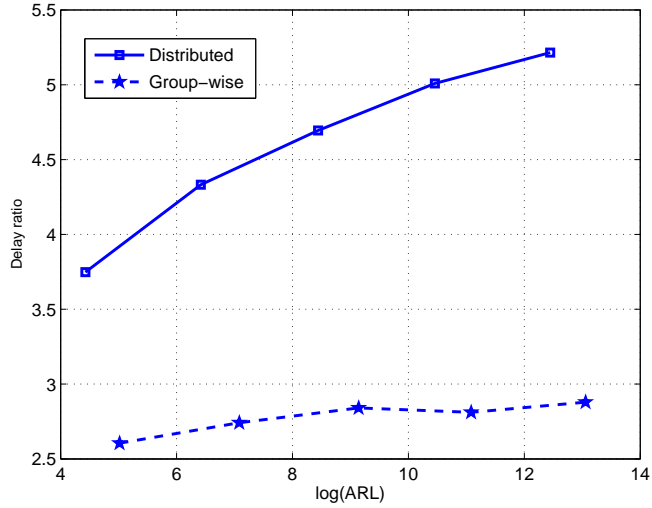}
\caption{Detection delay ratio for the continuous-time Brownian motion case}
\label{fig:continuous}
\end{figure}
Figure~\ref{fig:continuous} illustrates the ratio between the detection delay of the proposed scheme and that of the case with $N-1=8$ honest sensor for the continuous-time Brownian motion case. Based on asymptotic results discussed in Section~\ref{sec:Brownian}, this ratio should not be larger than
$$\frac{4/\mu^2}{2/[(N-1)\mu^2]}=2(N-1)=16$$
for the distributed scheme, and not larger than
$$\frac{12/(N\mu^2)}{2/[(N-1)\mu^2]}=6\frac{N-1}{N}=16/3$$
for the group-wise scheme. From the figure, we can see that the ratios are indeed smaller than these numbers. Furthermore, compared with the upper-bound obtained in this paper, we can see from the figure that the actual performance of the proposed schemes is even better in practice.

\begin{figure}[thb]
\centering
\includegraphics[width=0.6 \textwidth]{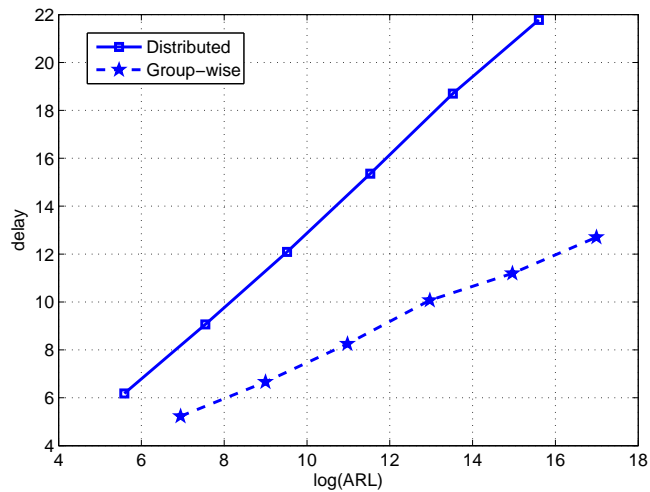}
\caption{Detection delay vs ARL for the discrete time case}
\label{fig:dis}
\end{figure}

Figure~\ref{fig:dis} illustrates the relationship between the detection delay and ARL for the discrete-time case. In generating this figure, we set the pre-change pdf $f_0$ to be Gaussian with mean $0$ and variance $1$, and set the after-change pdf $f_1$ to be Gaussian with mean $1$ and variance $1$. Same as the continuous-time case, we set the number of sensors $N$ to be 9. From the figure, one can again see that the detection delay scales logarithmically with ARL for both distributed and group-wise schemes. Furthermore, we can also see that the delay and its slope of the group-wise scheme are smaller than those of the distributed scheme.

\bibliographystyle{siam}
\bibliography{macros,sensornetwork,secrecy}

\end{document}